\documentclass{amsart}
\usepackage{lmodern}
\usepackage[utf8]{inputenc}
\usepackage[T1]{fontenc}

\usepackage{amsmath}
\usepackage{amssymb}
\usepackage{amsthm}
\usepackage{enumerate}
\usepackage{bbm}

\usepackage{multicol}
\usepackage{amsfonts}
\usepackage[colorlinks=true,citecolor=red,linkcolor=blue]{hyperref}
\usepackage{url}

\newtheorem{theorem}{Theorem}[section]
\newtheorem{proposition}[theorem]{Proposition}
\newtheorem{corollary}[theorem]{Corollary}
\newtheorem{lemma}[theorem]{Lemma}
\newtheorem{problem}[theorem]{Problem}
\newtheorem{claim}[theorem]{Claim}

\newtheorem*{theorem*}{Theorem}
\newtheorem*{corollary*}{Corollary}

\theoremstyle{definition}
\newtheorem{definition}[theorem]{Definition}
\newtheorem{rmk}[theorem]{Remark}

\newcommand{\al}{\omega_1}

\renewcommand{\P}{\mathbb {P}}

\newcommand{\Q}{\mathbb {Q}}
\newcommand{\dQ}{\dt{\mathbb {Q}}}
\newcommand\restr[2]{{
  \left.\kern-\nulldelimiterspace 
  #1 
  \littletaller 
  \right|{#2} 
  }}

\newcommand{\littletaller}{\mathchoice{\vphantom{\big|}}{}{}{}}

\DeclareMathOperator{\supt}{supt}

\newcommand{\seb}{\subseteq}

\newcommand{\res}{\!\upharpoonright\!}
\newcommand{\tup}[1]{\langle#1\rangle}

\newcommand{\OGA}{\mathrm{OGA}}
\newcommand{\ZFC}{\mathrm{ZFC}}
\newcommand{\dt}[1]{\dot{#1}}
\newcommand{\forces}{\Vdash}

\DeclareMathOperator{\hht}{ht}

\DeclareMathOperator{\lex}{lex}

\DeclareMathOperator{\tp}{tp}
\newcommand{\stp}{\tp_{\perp}}
\newcommand{\qtext}[1]{\text{``#1''}}

\title{Entangled Suslin lines and OGA}

\author{Carlos Martínez-Ranero}
\address[Martínez-Ranero]{Department of Mathematics, Universidad de Concepción, Concepción, Chile}
\email{cmartinezr@udec.cl}
\urladdr{www2.udec.cl/\textasciitilde cmartinezr}

\author{Lucas Polymeris}
\address[Polymeris]{Department of Mathematics, Universidad de Concepción, Concepción, Chile}
\email{l.polymeris@proton.me}

\thanks{MSC 2020: <completar>}
\thanks{\emph{Keywords and}: Entangled sets, Suslin lines, Forcing.}
\thanks{The first named author was partially supported by Proyecto VRID-Investigación  No. 220.015.024-INV}

\begin{document}

\begin{abstract}
We construct a model of the Open Graph Axiom (OGA) in which there is  a 2-entangled Suslin line $S$. Consequently, in this model, there is a 2-entangled uncountable linear order, but no such order is separable. This resolves a problem posed by  Carroy, Levine, and Notaro \cite{carroy2025} and answers a question from McKenney on MathOverflow \cite{Mckenney2014}.
\end{abstract}
\maketitle


\section{Introduction and main results}

Recall that an uncountable linear order $L$ is called $2$-entangled if
for every  pairwise disjoint sequence  $\{(x_{\xi},y_{\xi}) : \xi < \omega_{1}\} \seb
[L]^{2}$, and any function $t : 2 \to \{<,>\}$ there exist distinct
$\xi,\eta<\al$ such that $x_{\xi} \mathrel{t(0)} x_{\eta}$ and
$y_{\xi} \mathrel{t(1)} y_{\eta}$.  The concept of an entangled set of reals
was introduced by  Abraham and  Shelah \cite{AS81} as a strong witness of
non-minimality: no $2$-entangled linear order contains an (uncountable)
minimal suborder.

Todorcevic \cite{Todorcevic89} observed that  any $2$-entangled
linear order is c.c.c., which naturally leads to the question of whether  a $2$-entangled
Suslin line can exist.  Krueger \cite{Krueger2020suslin} answered this affirmatively by constructing  a c.c.c. forcing to add such a
Sulin line. Later,  Carroy, Levine, and Notaro \cite{carroy2025}
constructed a $2$-entangled Suslin line from $\diamondsuit$ and
asked whether it is consistent that there is a $2$-entangled uncountable linear
order, but no such order is separable. A natural approach to this problem is to  construct a model where the Open Graph Axiom ($\OGA$) holds
while preserving a $2$-entangled Suslin line.  Whether such a model exists was asked by McKenny in MathOverflow \cite{Mckenney2014}.

In this paper, we show that this is indeed possible,  assuming the existence of a supercompact cardinal. The proof is divided into five main steps, which we outline here to guide the reader. We begin with the $2$-entangled Suslin line constructed by Krueger, which is presented as a lexicographically ordered Suslin tree,
$(S,<_{S},<_{\lex})$.

\begin{itemize}\itemsep0em
  \item
  \emph{Step 1}: We show that
  $S$ satisfies a stronger
  form of $2$-entangledness, which we call weakly
  bi-entangledness. This property depends not only on the linear order, but also on the underlying Suslin tree structure.
  This stronger form is essential for iterating certain forcings while preserving that $S$ is $2$-entangled. (Section~\ref{sec:wbe})

  \item \emph{Step 2:} We define a strengthening of  properness for forcing posets, inspired
  by the works of Shelah \cite{Shelah98}, Schlindwein \cite{Schlindwein94} and Krueger \cite{Krueger2020}. We call this property $E_{S}$-\emph{properness}. We
  show that if $\P$ is $E_{S}$-proper and
  forces that $(S,<_S)$ remains Suslin, then it also
  forces that $(S,<_S,<_{\lex})$ remains  weakly bi-entangled. (Section~\ref{sec:ET-proper})

  \item \emph{Step 3:} To handle all relevant instances of $\OGA$ via iteration, we need forcings that are both $E_{S}$-proper
  and  preserve that $(S,<_S)$ is a Suslin tree.
  We show this holds for a variant of the forcings for $\OGA$
  used in \cite{Todorcevic2011}, which are already known to preserve a given Suslin tree. (Section~\ref{sec:forcing-OGA})

  \item \emph{Step 4:} We prove a preservation theorem for countable support iterations. While Miyamoto's Theorem \cite{Miyamoto93} shows that such iterations preserve a given Suslin tree if each individual poset does, the difficult part is to prove that
  countable support iterations of $E_{S}$-proper forcings remain $E_{S}$-proper. This is the most technical part of the paper. (Section~\ref{sec:iteration})
  \item \emph{Step 5:} Finally, in Section~\ref{sec:final-model}, we start from a supercompact cardinal (which will be used as a bookkeeping device) and perform a countable support iteration to handle  all the relevant  instances of $\OGA$. This yields  a model where OGA holds and $S$ remains a $2$-entangled Suslin
  line.
\end{itemize}

\section{Preliminaries}\label{sec:prelim}
In this section, we will review the necessary background  and establish some notation that will be used throughout the rest of the paper.
\subsection{Trees and lexicographic orderings}

A  set-theoretic tree is a partially ordered set $(T,<_{T})$  such that
for any node $t \in T$, the set of predecessors  $\{s \in T : s <_{T} t\}$   is   well-ordered by $<_T$. The \emph{height} of $t$, written $\hht_{T}(t)$ is the unique ordinal representing the  order-type of the set of its predecessors.  For an ordinal $\alpha$, the
\emph{$\alpha$-th} level of $T$, written $T_{\alpha}$, consists
of the elements of $T$ of height $\alpha$. The \emph{height of $T$}
is the least ordinal $\alpha$ for which $T_{\alpha}$ is empty.

A \emph{branch} of $T$ is any downward closed
chain of $T$, and a branch is called \emph{cofinal} if its height is equal to the height
of $T$. For each $t \in T$ and $\xi \le \hht(t)$, we let
$t\res \xi$ be the unique $s \in T$ of height $\xi$ such that
$s \le_{T} t$. If the tree is clear from the context, then $t \perp s$ denotes the fact  that $t$ and $s$ are incomparable in the poset $(T,\le_{T})$. The tree $T$ is said to be \emph{well-pruned} if
 all nodes $t \in T$ have uncountably many extensions.

For incomparable $t,s \in T$, we let
$\Delta(t,s)$  be the least ordinal  $\alpha$ such that
$t \res \alpha \neq s \res \alpha$. We leave $\Delta$ undefined
if $t$ and $s$ are comparable.

An $\omega_{1}$-tree is a tree of height $\omega_{1}$ with
all its levels countable. An  $\al$-tree is called an \emph{Aronszajn
tree} if it has no uncountable chains, equivalently, no cofinal branches.
An $\al$-tree is called a \emph{Suslin tree} if it has no uncountable chains and no uncountable
antichains. A typical argument shows that if $T$ is a well-pruned
$\omega_{1}$-tree, then it is Suslin if and only if  it has no uncountable
antichains.

We will now review some  properties of linear orders which are needed for our main results.
\begin{definition} Let $L$ be a linearly order set.
\begin{enumerate}[{(1)}]
    \item $L$ is \emph{dense} if for any $a<_Lb$ there exists $c$ such that $a<_L  c <_L b$.
    \item $L$ has the \emph{countable chain condition} or \emph{c.c.c.}, if any family of pairwise disjoint nonempty open intervals is countable.
    \item $L$ is \emph{order-theoretic separable} if there exists a countable set $D$ such that for any $a<_L b$ there is a $d\in D$ such that $a\leq_L d\leq_L.$
    \item $L$ is \emph{topologically separable} if its separable in the order topology, i.e., there is a countable set $D$
such that $D$ intersects every nonempty open interval.
\end{enumerate}

\end{definition}
It is easy to see that a linear order $L$ is order-theoretic separable iff it is isomorphic to a suborder of the reals.
\begin{rmk}
    Notice that the two notions of separable agree on dense linear orders. However, they disagree in general. For example, the double arrow of Alexandroff is topologically separable but not order-theoretic separable.
\end{rmk}

 Observe that any topologically separable linear order is c.c.c.
 \begin{definition}
     A \emph{Suslin} line is a linear order which is c.c.c. and not topologically separable.
 \end{definition}

There is a strong relationship between trees and linearly ordered sets, given by the following procedure.

Assume that $T$ is an Aronszajn tree and that for any $\alpha<\al$, we have a linear order   $\triangleleft_{\alpha}$
of $T_\alpha$.  Then this induces a \emph{lexicographic}
ordering of $T$ by letting
$t <_{\lex} s$ if either $t <_{T} s$ or $t \perp s$ and
$t\res \Delta(t,s) \triangleleft_\alpha s\res \Delta(t,s)$ where $\alpha=\Delta(t,s)$. This is always
a linear order on $T$ extending $\le_{T}$. Any order obtained in this
fashion is called a \emph{lexicographic ordering} of $T$.
Moreover, any Suslin line contains an uncountable suborder isomorphic
to a lexicographically ordered Suslin tree, and any lexicographic
ordering of a Suslin tree contains a Suslin line.

For more information on trees, the reader can consult \cite{Todorcevic1984}. Also, we
do not assume that our trees are
Hausdorff (i.e. nodes of limit height are not necessarily determined by their predecessors).
\subsection{Entangled linear orders}

Let $L$ be an uncountable linear order and $m\in\omega$. Given $a\in [L]^m$ we will identify it with the sequence $\{a(i):i<m\}$  given by its increasing enumeration. We say that a pairwise disjoint sequence $\{a_\xi:\xi<\al\}\subseteq [L]^m$ is {\emph separated} if there is a $c\in [L]^{m-1}$ such that $ a_\xi(i)<c(i)<a_{\xi}(i+1) $ for $i<m-1$. By a {\emph type} we mean a function $t:m\to \{<,>\}.$
\begin{definition}
    Let $L$ an uncountable linear order, $m\in\omega,$ and $t:m\to\{<,>\}$ be a type and $a,b\in [L]^m$ disjoint.
    \begin{enumerate}[{\rm (1)}]
        \item We say that $(a,b)$ \emph{ realizes} the type $t$ if $a(i) t(i) b(i)$ for $i<m$.
        \item We denote by $\tp(a,b)$ the only type realized by $(a,b).$
    \end{enumerate}
\end{definition}
We can now define the notion of $m$-entangled sets and some variations of it.
\begin{definition}
Let $L$ be an uncountable linear order and $m\in\omega$.
\begin{enumerate}[{\rm (1)}]
    \item $L$ is $m$-\emph{entangled} if for any pairwise disjoint sequence $\{a_\xi:\xi<\al\}\subseteq [L]^m$ and any type $t:m\to \{<,>\}$ there exists $\xi\ne \eta<\al$ such that $\tp(a_\xi,a_\eta)=t$.
    \item $L$ is \emph{weakly} $m$-\emph{entangled} if for any separated pairwise disjoint sequence $\{a_\xi:\xi<\al\}\subseteq [L]^m$ and any type $t:m\to \{<,>\}$ there exists $\xi\ne \eta<\al$ such that $\tp(a_\xi,a_\eta)=t$.
\end{enumerate}
\end{definition}

Entangled linear orders are very interesting combinatorial objects with strong topological properties. For example,  Todorcevic \cite{Todorcevic85} points out without a proof  that any 2-entangled linear order is c.c.c. and any 3-entangled linear order is topologically separable (for proofs of the aforementioned facts see \cite{Krueger2020suslin}).


We will need the following
result, which is proved in
\cite[Proposition 3.8]{Krueger2020suslin}.

\begin{proposition}\label{prop:dense-w2e-2e}
  Let $(S,<_{S},<_{\lex})$ be a lexicographically ordered Suslin
  tree such that $(S,<_{\lex})$ is a dense linear order. Then
  $(S,<_{\lex})$ is $2$-entangled if and only if it is weakly $2$-entangled.
\end{proposition}
\subsection{The Open Graph Axiom}
Let $X$ be a separable metric space. A \emph{graph} on $X$ is a structure of the form $G:=(X,E)$ where $X$ denotes  the set of \emph{vertices} and $E$ is an irreflexive, symmetric subset of $X^2$ is the set of \emph{edges}. We say that $G$ is an \emph{open graph} if the set $E$ is open in $X^2$ with the product topology. A subset $Y\subseteq X$ is called a \emph{complete subgraph} if $[Y]^2\subseteq E$. A subset $Y\subseteq X$ is called  \emph{independent } if $[Y]^2\cap E=\emptyset$. A graph $G=(X,E)$ is \emph{countably chromatic} if there is a countable decomposition $X=\bigcup_{n\in\omega}X_n$ of its vertex set such that $X_n$ is independent for all $n\in\omega.$ Notice that, trivially, if $X$ has an uncountable complete subgraph, then it cannot have a countable  chromatic number.

Todorcevic \cite{Todorcevic89} introduced the following dichotomy for open graphs, called the \emph{Open Graph Axiom}.
\begin{definition}
    \textbf{OGA}: Let $X$ be a separable metric space and $G=(X,E)$ be an open graph. Then one of the following alternatives holds.
    \begin{enumerate}[{\rm (1)}]
        \item $G$ is countably chromatic or
        \item $G$ has an uncountable complete subgraph.
    \end{enumerate}
\end{definition}

The following Proposition is a well-known result of Todorcevic, but we prove it here for the sake of completeness and  future reference.

\begin{proposition}
    If there is a 2-entangled topologically separable linear order, then OGA fails.
\end{proposition}
\begin{proof}
    Let $L$ be a 2-entangled topologically separable linear order. By going to an uncountable suborder, if necessary, we may assume that $L$ is order-theoretic separable. Thus, $L$ is (isomorphic to) a suborder of the reals. Consider the  open graph with vertices $L^2$ and edge set $E$ given by $$E=\{\{(x_0,x_1),(y_0,y_1)\}\in [L]^2:  x_0<_Ly_0 \ \text{iff}\ x_1<_L y_1\}.$$
    We will prove that both alternatives of OGA fail. To see this, fix an uncountable subset $Y\subseteq L^2$. Let $t_0$ to be any constant type i.e.  any type that satisfies $t_0(0)= t_0(1)$ and let $t_1$ be any non constant type i.e. any type that satisfies   $t_1(0)\ne t_1(1)$. Since $L$ is 2-entangled and $Y$ is uncountable, then there exists $(a,b), (c,d)\in [Y]^2$ such that  $\tp(a,b)=t_0$ and $\tp(c,d)=t_1$.
    First, notice that $Y$ is not an independent subset since $(a,b)\in E$ (by definition of $G$).  This implies that $G$ is not countably chromatic (since no uncountable subset is independent).  Next, observe that $Y$ is not a complete subgraph since $(c,d)\not\in E$ (by definition).  Thus, both alternatives of OGA fail.

\end{proof}

\section{Weakly bi-entangledness}\label{sec:wbe}
In this section, we establish several fundamental properties of weakly entangled lexicographically ordered Aronszajn trees. We then introduce a stronger notion, termed weakly bi-entangled, and conclude by proving the consistency of the existence of a weakly bi-entangled Suslin line.

Before going any further, we need to add some notation.

Let $(T,<_T,<_{\lex})$ be a lexicographically ordered Aronszajn tree.
 To avoid ambiguities we will always
use $<_{T}$ to refer to the tree order, and $<$ or $<_{\lex}$ to refer to  the
lexicographic ordering. If $a,b \in [T]^{2}$ and $t : 2 \to \{<,>\}$
is a type,
we will use $\stp(a,b) = t$ to abbreviate the fact that
$\tp(a,b) = t$ and that $a(i) \perp b(i)$
for all $i < 2$.

We will define a strengthening  of weakly 2-entangled, which we call weakly bi-entangled, inspired by the work of Miyamoto and Yorioka \cite{MiyamotoYorioka20}. We will show that this stronger version is well behaved under  countable support iterations.
\begin{definition}\label{df:wk-bi-e}
  We say that $(T,<_T,<_{\lex})$ is \emph{weakly bi-entangled} if for every
  uncountable pairwise disjoint family  $A \seb [T]^{2}$ which is separated there
  is an $\xi_0 < \omega_{1}$ such that for all $b \in A $ with $\min\{\hht(a(0)),\hht(a(1))\}\ge\xi_0 $
  and every type $t : 2 \to \{<,>\}$,
  there is $a \in A \cap [T\res \xi_0]^{2}$ such that
  $\stp(a,b) = t$.
\end{definition}
It follows from the definitions  that if $T$ is weakly bi-entangled, then $T$ is weakly $2$-entangled.
The following definition is  key for our preservation results.
\begin{definition}\label{def:normal}
    Let $(T,<_T,<_{\lex})$ be a lexicographically ordered Aronszajn tree. A subset $B\subseteq [T]^2$ is called \emph{normal} if the following conditions hold:
    \begin{itemize}
        \item $B$ is an uncountable collection of pairwise disjoint sets.
        \item For every $b\in B,$ $\hht_T(b(0))=\hht_T(b(1))$, we denote its common value by $\hht_T(b)$.
        \item The set $\Gamma:=\{\Delta(b): b\in B\}$ is bounded and $\sup(\Gamma)<\min(\hht_T(B)).$
        \item The set $\{\hht_T(b): b\in B\}=[\min(\hht_T(B)),\al).$
    \end{itemize}

\end{definition}
The next Lemma allows us to reduce the verification of weakly bi-entangledness to  normal sequences.

\begin{lemma}\label{lem:wbe-eq}
  Let $T$ be a lexicographically ordered Aronszajn tree. The following
  are equivalent:
  \begin{enumerate}[(a)]
    \item $T$ is weakly bi-entangled.
    \item For every normal $A \seb [S]^{2}$ and $t : 2 \to \{<,>\}$ there
    are $a,b \in A$ such that $\hht(a) < \hht(b)$ and $\stp(a,b) = t$.
    \item For every uncountable pairwise disjoint and separated
    $A \seb [S]^{2}$, there are $a,b \in A$ such that
    $\max\{\hht(a(0)),\hht(a(1))\} < \min\{\hht(b(0)),\hht(b(1))\}$ and
    $\stp(a,b) = t$.
  \end{enumerate}
\end{lemma}
\begin{proof}
  $(a) \rightarrow (b)$ It is clear since every normal
  family contains an uncountable subfamily which is separated.

  $(b) \rightarrow (c)$. Let $A \seb [S]^{2}$ be uncountable pairwise
  disjoint and separated. Since $A$ is separated and
  pairwise disjoint, then there is a $\xi_{0} < \omega_{1}$ such that
  $$A' := \{a \in A : a(0) \perp a(1), \Delta(a(0),a(1)) < \xi_{0},
  \hht(a(0)),\hht(a(1)) \ge \xi_{0}\}$$ is uncountable.
  We will recursively define a sequence
  $\{a_{\xi} : \xi_{0} \le \xi < \omega_{1}\} \seb A'$,
  satisfying the following conditions:
  \begin{enumerate}[(1)]\itemsep0em
    \item For every $\xi$, $\xi \le \min\{\hht(a_{\xi}(0)),\hht(a_{\xi}(1))\}$.
    \item For every $\xi < \eta$,
    $\max\{\hht(a_{\xi}(0)),\hht(a_{\xi}(1))\}
    < \min\{\hht(a_{\eta}(0)),\hht(a_{\eta}(1))\}$
  \end{enumerate}
  Suppose that $\{a_{\xi} : \xi_{0} \le \xi < \gamma\}$ has been defined
  for some $\xi_{0} \le \gamma < \omega_{1}$. Using the fact that $A$ is pairwise
  disjoint property, find an $a \in A'$ such that
  $\hht(a(0)),\hht(a(1)) > \sup(\{\gamma\} \cup
  \{\max\{\hht(a_{\xi}(0)),\hht_{\xi}(a(1))\} + 1 :
  \xi < \gamma\})$. Let $a_{\gamma} = a$.

  Now, for each $\xi_{0} \le \xi < \al$, let
  $b_\xi:=\{a_\xi(0)\res \xi,a_\xi(1)\res \xi\}$. Observe that
  $a_{\xi}(0) \res \xi \neq a_{\xi}(1) \res \xi$, since
  $\Delta(a(0),a(1))<\xi_0$ for any $a\in A.$

  Finally, let $B := \{b_{\xi} :\xi_0\leq \xi<\al\}$.
  It follows that $B$ is a normal sequence. Thus, by hypothesis, there are
  $\xi_{0} \le \xi< \eta$ such that $\stp(b_{\xi},b_{\eta}) = t$.
  Since for $i<2$,
  $b_{\xi}(i) \bot b_{\eta}(i)$, it follows that
  $\stp(a_{\xi},a_{\eta}) = t$. By construction we have
  that $\max\{\hht(a_{\xi}(0)),\hht(a_{\xi}(1))\} <
  \min\{\hht(a_{\eta}(0)),\hht(a_{\eta}(1))\}$ so we are done.

  $(c) \rightarrow (a)$. We prove the contrapositive.
  Suppose that $T$ is not weakly by entangled. Then there is
  an uncountable pairwise disjoint and separated $A \seb [S]$ such that
  for all $\alpha < \omega_{1}$, there is $b_{\alpha} \in A$ and
  type $t_{\alpha}$ such that
  $\hht(b_{\alpha}(0)),\hht(b_{\alpha}(1)) \ge \alpha$, and
  $\stp(a,b_{\alpha}) \neq t_{\alpha}$ for all $a \in A \cap [S \res \alpha]^{2}$.
  Then there is some type $t$ and
  uncountable $\Gamma \seb \omega_{1}$ such that
  $t_{\alpha} = t$ for all $\alpha \in \Gamma$. Using the pairwise disjoint
  property one easily finds an uncountable subset $\Gamma' \seb \Gamma$ such that
  for all $\alpha < \beta$ in $\Gamma'$,
  $\hht(b_{\alpha}(0)),\hht(b_{\alpha}(1)) < \beta$. This  implies
  that $\{b_{\alpha} : \alpha \in \Gamma'\}$ witnesses the failure of
  (c).
\end{proof}

\subsection{Forcing a weakly bi-entangled line}
As we mentioned before,  Krueger constructed a 2-entangled Suslin line.
This section is devoted to prove that Krueger's Suslin line  is actually weakly bi-entangled. In order to do so, we need to recall some notation and review part of his construction.

Let $\mathbb{Q}$  denote the rationals. For each $\alpha<\al$, let $\Q_\alpha:= \{(\alpha,q) : q \in \Q\}$. Define $(\alpha,q) <_\alpha (\alpha,r)$ if $q < r$ in $\Q$.  Let $\Q^* :=\bigcup \{\Q_\alpha : \alpha<\al \}$. Define a function $h$ on $\Q^*$ by $h(\alpha, q) := \alpha$ for all $(\alpha, q) \in  \Q^*$. The forcing poset $\P$ we will define introduces a Suslin tree $S$ such that for all $\alpha < \al, \Q_\alpha$ is equal to the level $\alpha$ of $S$.

\begin{definition} Let  $\P$ denote the forcing poset whose elements are finite trees $p$ whose nodes belong to $\Q^*$ and satisfy that $x <_p y$ implies $h(x) < h(y)$.  Let $q \leq p$ if  the underlying set of $p$ is a subset of the underlying set of $q$ and for all $x$ and $y$ in $p, x <_p y$ iff $x <_q y$.
\end{definition}

Let $\dt{S}$ be the canonical name for the partial
order given by $\bigcup_{p \in \dt{G}}<_{p}$, where $\dt{G}$ is the
canonical name for a $\P$-generic filter. Krueger showed that
$\dt{S}$ it is forced to be a well-pruned Hausdorff Suslin tree
with underlying set equal to $\Q^{*}$.
It is also forced that for each $s \in S$,
the set of its  immediate successors
is a subset of $\Q_{\alpha}$ (for some $\alpha$), and therefore $\dt{S}$ can
be lexicographically ordered using the orders $<_{\alpha}$. An
easy genericity argument shows that $(\dt{S},<_{\lex})$ is a densely ordered Suslin line.

We will need a few more facts from \cite{Krueger2020suslin} about
the forcing $\P$.

The next Lemma appears as \cite[Lemma 2.7]{Krueger2020suslin} and it will be used to prove compatibility of conditions in $\P.$
\begin{lemma}\label{Kruegercompatibility}
    Let $p \in \P$. Suppose that $\{(a_i,b_i) : i < n\}$ is a family of distinct pairs, where $0 < n < \omega$, such that for each $i < n, a_i$ is the immediate predecessor of $b_i$ in $p$. Let $x_0,\dots,x_{n-1}$ be distinct members of $\Q^*$ satisfying that for all $i < n, h(a_i) < h(x_i) < h(b_i)$. Then there exists $q \leq p$ with underlying set equal to the underlying set of $p$ together with $x_0,\dots,x_{n-1}$ such that for all $i < n, a_i <_q x_i <_q b_i$.
\end{lemma}

The following Lemma appears as \cite[Lemma 2.11]{Krueger2020suslin} and it will be helpful to compare  nodes via the induced lexicographical ordering.

\begin{lemma}\label{kruegerlexorder}
    Assume that $p \in \P, x, y, z, a$, and $b$ are distinct members of $p$, and
    \begin{enumerate}[{\rm (i)}]
        \item $x<_p y<_p a$;
        \item  $x<_p z<_p b$;
        \item $h(y)=h(z)=h(x)+1$;
        \item $ y <_{h(x)+1} z$.
    \end{enumerate}

Then $p$ forces that $a <_{\dt{S}}b$.
\end{lemma}

Finally, we are in position to prove  that the $2$-entangled Suslin line introduced by Krueger is actually weakly bi-entangled.

\begin{theorem}\label{thm:weaklybientangled}
  The forcing poset $\P$  adds a Suslin line which is weakly bi-entangled.
\end{theorem}
\begin{proof}
We will proof that $\P$ forces that  $(\dt{S},<_S,<_{\lex})$ is weakly bi-entangled. Fix a type $t:2\to \{<,>\}$.  Suppose $p\in \P, c\in\mathbb{Q}^*$ and $p$ forces that $\dt{A}$ is an  an uncountable pairwise disjoint  subfamily  of $[\dt{S}]^2$  separated by $c$.

Let $M\prec H(\omega_2)$ be a countable elementary submodel which contains $\P,c, \dt{A} $ and $p$ as elements. Let $\delta:=M\cap\al.$

 We claim that $\delta$ is as required. In order to see this, fix a further extension $q$ of $p$ and $b\in [S]^2$ with the minimum  height of its coordinates bigger than $\delta$ such that $q\forces ``b\in \dt{A}".$ If necessary,  by further extending $q$, we may assume without loss of generality that $b\in q$. Our goal is to find $s\leq q$ and $a\in [S\restriction\delta]^2$ such that $s\forces`` a\in\dt{A} \wedge \stp(a,b)=t"$. To do this,  let $q_M:=q\cap M,$ and let $b^*(i)$ the minimal node in $q\setminus \delta $ below or equal $b(i)$ for each $i<2$. Let $B$ be the set of all conditions $r$ such that:
\begin{enumerate}[{\rm (1)}]
    \item There exists a limit ordinal $\xi$ such that $r\restriction\xi=q_M$;
    \item There exists $a\in [S]^2$ such that $a\in r\setminus \xi$ and $r\forces`` a\in\dt{A}"$;
    \item If $a^*(i)$ denote the minimal node in $r\setminus \xi $ below or equal $a(i)$, then $\{x\in q_M: x<_ra^*(i)\}=\{x\in q_M: x<_rb^*(i)\}$ for $i<2$.
\end{enumerate}
Observe that $a^*(0)\ne a^*(1)$ since $q\forces``a(0)<_{\dt{S}} c<_{\dt{S}} a(1)"$ and both have height bigger than $\delta.$ Also note that $B$ is in $M$ as is definable from parameters in $M$. By elementarity, since $q\in B$, then there exists $r\in \P\cap M.$ Let $a\in[S]^2\cap M$ be a witness that $r\in B.$  By assumption, for each $i<2$,  $a^*(i) $ and $b^*(i)$ have the same immediate predecessor $d_i\in r\restriction\xi=q\restriction\delta=q_M.$ Since $\xi$ is a limit ordinal $h(d_i)<\xi\leq h(a^*(i)),h(b^*(i))$ for $i<2$. For each $i<2$, choose rational numbers  $q_{b,i}$ and $q_{a,i}$ such that none of them appear in $r\cup q$ and $q_{a,i} t(i) q_{b,i}$. Define $x_i:=(h(d_i)+1,q_{a,i})$ and $y_i:=(h(d_i)+1,q_{b,i})$ in $\mathbb{Q}_{h(d_i)+1}$ for all  $i<2$. Using Lemma \ref{Kruegercompatibility}, we can find an extension $s\leq q,r$ such that for all $i<2$, $d_i<_s x_i<_s a^*(i)$, $d_i<_s y_i<_s b^*(i)$, $h(y_i)=h(x_i)=h(d_i)+1$ and $x_i t(i)y_i$. It follows that $s$ forces that $a(i)$ and $b(i)$ are $<_{\dt{S}}$ incompatible. Applying Lemma   \ref{kruegerlexorder}, we obtain that $s$ forces that $\stp(a,b)=t$  and $a,b\in \dt{A}$ as required. This concludes the proof of the Theorem.
\end{proof}

\section{Preserving 2-entangled Suslin lines}\label{sec:ET-proper}
In this section we identify a class  of proper forcings that preserve that our   Suslin line $S$ is 2-entangled. Proposition \ref{prop:dense-w2e-2e} shows that for a dense Suslin line, being 2-entangled is equivalent to being weakly 2-entangled. Consequently, preserving being 2-entangled reduces to two distinct problems: ensuring $S$ remains Suslin and ensuring it remains weakly 2-entangled in the extension.

\subsection{Preserving Suslin trees }
In this subsection we will focus on the task of preserving a Suslin line. In order to do this, we will need the following notion.
\begin{definition}
  Let $S$ be a Suslin tree and $\P$ a forcing notion.
  $\P$ is said to be $S$-\emph{preserving} if
  $\forces_{\P} \qtext{$S$ is Suslin}$.
\end{definition}

The preservation of a Suslin tree under countable support iterations is a consequence of the following well-known Theorem of Miyamoto \cite{Miyamoto93}.

\begin{theorem}\label{miyamoto}
  Let $S$ be a Suslin tree in the ground model.
  Then any countable support iteration of $S$-preserving
  proper forcings is again $S$-preserving (and proper).
\end{theorem}

\subsection{Preserving weakly bi-entangled}
In this subsection we will focus on the second task  meaning  preserving being weakly 2-entangled. To do this, we will introduce a new class of proper forcings. Before proceeding any further we need to introduce some definitions and fix some notation.

For this subsection fix $(T,<_{T},<_{\lex})$ a lexicographically
ordered Aronszajn tree. We will often
write $T$ instead of $(T,<_{T},<_{\lex})$. Since
we deal only with binary types, by a type we will always mean
some function $t : 2 \to \{<_{\lex},>_{\lex}\}$.

\begin{definition}
  Assume $T$ is weakly bi-entangled.
  We say that $\P$ is \emph{$T$-weakly bi-entangled preserving} if
  $\forces_{\P} \qtext{$T$ is weakly bi-entangled}$.
\end{definition}

Since the notion of weakly bi-entangled preserving can be hard to
work with, we will use a more technical condition, inspired by the work of Schlindwein\cite{Schlindwein94}, which we will show to be equivalent
to weakly bi-entangled preserving for proper forcings.

\begin{definition}
  Let $N$ be a countable set such that $\delta := N \cap \omega_{1}$
  is an ordinal. We say that $b \in [T_{\delta}]^{2}$ is
  \emph{$(N,E_{T})$-generic} if for any $A \seb [T]^{2}$ which is a
  member of $N$
  and any type $t$, if $b \in A$ then there is $a \in N \cap A$ such that
  $\stp(a,b) = t$.
\end{definition}

\begin{lemma}\label{lemma:nelementary-implies-bproper}
  Let $\theta$ be a large enough regular cardinal,
  and let $N$ be a countable elementary submodel of $H(\theta)$ such
  that $T \in N$. If $T$ is weakly bi-entangled,
  then every $b \in [T_{N \cap \omega_{1}}]^{2}$ is $(N,E_{T})$-generic.
\end{lemma}
\begin{proof}
  Let $\theta$, $T$, $N$, $A$  $b$ and $t$ be as in the statement of the Lemma and suppose that $b\in A$. Let
  $\delta := N \cap \omega_{1}$. Our goal is to find $a\in A\cap N$ such that $\stp(a,b)=t$. To do this, first observe that since
  $\hht(b(0)) = \hht(b(1))$ and $b(0) <_{\lex} b(1)$, then there
  is an $x \in N$ such that $b(0) <_{\lex} x <_{\lex} b(1)$ (for
  example, we may take  $x:= b(1) \restriction \xi$ for any $\Delta(b(0),b(1)) < \xi <
  \delta$ ).
  Let $B := \{a \in A : \hht(a(0)) = \hht(a(1)),
  a(0) <_{\lex} x <_{\lex} a(1)\}$. Notice that $B\in  N$ since it is definable from parameters in $N$. Also observe that since $b\in B \setminus N$, it follows that  $B$ is uncountable and
  separated. By definition of weakly bi-entangled
  (Definition~\ref{df:wk-bi-e}), there is a $\xi_0$ witnessing this property, by elementarity we may choose  $\xi_0$ to be a member of $N$. Hence,
  there exists $a \in B \cap [S\res \xi_0]$ such that $\stp(a,b) = t$, since
  $B \seb A$ and $\xi_0<\delta$, we conclude that $b$ is $(N,E_{T})$-generic.
\end{proof}

\begin{definition}
  Let $\P$ be a forcing notion, and $\theta$ a regular
  cardinal such that $T,\P \in H(\theta)$. Let $N$ be a countable
  elementary submodel of $H(\theta)$ with $\P \in N$.
  A condition $q \in \P$ is called \emph{$(N,\P,E_{T})$-generic} if
  it is $(N,\P)$-generic and
  whenever $\dt{A}$ is a $\P$-name in $N$,
  $b \in [T_{N\cap \omega_{1}}]^{2}$ is $(N,E_{T})$-generic,
  and there is an $r \leq q$  such that $r \forces`` b \in \dt{A} \seb [T]^{2}"$,
  then there exists $a \in [T]^{2} \cap N$ and $s \le r$
  such that $\stp(a,b) = t$ and $s \forces ``a \in \dt{A}"$.
\end{definition}

\begin{definition}
  We say that $\P$ is \emph{$E_{T}$-proper} if for all large enough regular cardinals
  $\theta$, there are club many countable  $N$ in   $[H(\theta)]^\omega$ such that $N$ is an elementary submodel of $H(\theta)$ with
  $\P,T \in N$ and for all $p \in \P \cap N$ there is a $q\leq p$
which is  $(N,\P,E_{T})$-generic.
\end{definition}

\begin{lemma}\label{lem:E-proper-aronszajn}
  Let $T$ be weakly bi-entangled and let $\P$ be $E_{T}$-proper forcing.
  Then $\forces_{\P} \text{``T is Aronszajn''}$.
\end{lemma}
\begin{proof}
  Let $G$ be any generic filter for $\P$, and suppose
  towards a contradiction that $b := \{b_{\xi} : \xi < \al\}$
  is a cofinal branch through $T$, such that for each $\xi$,
  $b_{\xi} \in T_{\xi}$. Choose $\alpha$ and $t \in T_{\alpha}$
  such that $\{s \in T : t \le_{T} s\}$ is uncountable
  and disjoint from $b$. For each $\alpha \le \xi < \al$,
  pick $a_{\xi}$ to be any element above $t$ of height $\xi$.
  It follows that $A := \{\{a_{\xi},b_{\xi}\} : \alpha \le \xi < \al\}$ is a pairwise disjoint  uncountable subset of $[T]^{2}$ which is separated. By symmetry, we may also assume
  that $a_{\xi} <_{\lex} b_{\xi}$ for all $\alpha \le \xi < \al$. In other words,
  if $a \in A$, then $a(1) \in b$.
  Let $\dt{A}$ and $\dt{b}$ be $\P$-names for $A$ and $b$, respectively
  and let $p$ be a condition in $G$ such that $p$ forces all the aforementioned properties about $A$ and $b$.
  Let $\theta$ be a large enough regular cardinal and let $N$ be
  a countable elementary submodel of $H(\theta)$ which contains
  $\dt{A},\dt{b},p,\P,T$ and $ \alpha$. Let $q$ be an $(N,\P,E_T)$-generic
  condition below $p$.  Find   $r\leq q$ and $b \in [T_{N \cap \omega_{1}}]^{2}$ such that
  $r \forces`` b \in \dt{A}"$. By definition of $(N,\P,E_T)$-generic, there exists $a \in [T]^{2} \cap N$
  and $s \le r$ such that $s \forces a \in \dt{A} \land \stp(a,b) = (>,>)$.
  In particular, $s \forces`` b(0) \bot b(1)"$ which is a contradiction
  since $s \le p$ and $p$ forces that $b$ is a branch.
\end{proof}

\begin{lemma}\label{lem:E-pres_E-proper}
  Let $T$ be weakly bi-entangled. Then $\P$ is $E_{T}$-proper
  iff $\P$ is $T$-weakly bi-entangled preserving and
  proper.
\end{lemma}
\begin{proof}
  Suppose that $\P$ is proper and $T$-weakly bi-entangled preserving.
  Let $\theta$ be a large enough regular cardinal and let $N$ be a countable elementary submodel of $H(\theta)$ which contains $T$ and $\P$ as elements. Fix $p \in N \cap P$ and let $q \le p$
  be $(N,\P)$-generic. We claim that this $q$ is
  $(N,\P,E_{T})$-generic. In order to see this, fix $b \in [T_{N \cap \omega_{1}}]^{2}$, $\dt{A}\in N$ a $\P$-name for a subset of $[T]^2$ and a type $t$.  Suppose $r \le q$ is such that
  $r \forces`` b \in \dt{A}"$.   Let $G\subseteq \P$ be any generic filter
  with $r \in G$. Since $\P$ preserves that $T$ is weakly bi-entangled, $b\in[T_{N[G]\cap\al}]^2\cap\dt{A}[G]$ (since $N[G]\cap \al=N\cap\al)$ and $N[G] \prec H(\theta)^{V[G]}$, then it follows, from Lemma \ref{lemma:nelementary-implies-bproper}, that there is an $a\in \dt{A}\cap N[G]$ such that $\stp(a,b)=t.$ Since $N[G]$ is a forcing extension of $N$ there is an $r'\in G\cap M$ such that $r'\forces ``a\in \dt{A}$. Since $r,r'\in G $, then there is an $s\leq r,r'$. Thus, $s$ is as required.

  (ii) Suppose that $\P$ is $E_{T}$-proper. Thus in particular $\P$ is
  proper, so it is enough to prove that $\P$ is weakly bi-entangled
  preserving. Let
  $\dt{B}$ be a $\P$-name, and $p \in \P$ a condition that forces that
  $\dt{B} \seb [S]^{2}$ is a normal sequence. By Lemma~\ref{lem:wbe-eq}
  it is enough to
  show that the set of conditions
  $s$ such that $s \forces a,b \in \dt{B}$ for
  some $a$ and $b$ satisfying $\hht(a) < \hht(b)$ and
  $\stp(a,b) = t$, is dense below $p$.

  To see this, let $p'$ be any condition extending $ p$. Let $\theta$ be a large enough regular cardinal and  let $N$ be a countable elementary
  submodel of $H(\theta)$ which contains $T,\P,p',\dt{B}$ as elements.
  Let $q \le p'$ be $(N,\P,E_{T})$-generic. Let $\delta := N \cap
  \omega_{1}$. Let $G$ be
  generic with $q \in G$. Recall that
  $N[G] \cap \omega_{1} = \delta$. Since $\dt{B}[G] \in N[G] \prec
  H(\theta)^{V[G]}$, and $\dt{B}$ is forced by $p$ to
  have an end-segment of levels, there is some
  $b \in \dt{B}[G] \cap [T_{\delta}]^{2}$. Fix
  $r \le q$ which forces this. Applying the
  definition of $(N,\P,E_{T})$-genericity, this tell us
  that for some $a \in [T]^{2} \cap N$ and $s \le r$, 
  $\stp(a,b) = t$ and $s \forces`` \dt{a} \in B"$. Since we already know that
  $r \forces`` b \in \dt{B}"$, and clearly $\hht(a) < \hht(b)$,
  this concludes the proof of the Lemma.
\end{proof}

From Proposition~\ref{prop:dense-w2e-2e} and Lemmas
\ref{lem:E-proper-aronszajn} and \ref{lem:E-pres_E-proper}, we obtain the
following.

\begin{corollary}\label{cor:E-proper-wbe}
  Suppose $S$ is a lexicographically ordered Suslin tree
  that is weakly bi-entangled, and such that
  $(S,<_{\lex})$ is dense. Let $\P$ be
  an $E_{S}$-proper forcing which  is also $S$-preserving.  Then
  $\forces_{\P}$ ``$(S,<_T,<_{\lex})$ is weakly bi-entangled
  and $(S,<_{\lex})$ is $2$-entangled''.
\end{corollary}

Finally, we will also need the following.

\begin{lemma}\label{lem:s-closed-ET-proper}
  Let $\P$ be an $\omega_{1}$-closed forcing notion.
  Then $\P$ is $E_{T}$-proper, and if $T$ is Suslin
  then $\P$ is also $S$-preserving.
\end{lemma}
\begin{proof}
  Let $\P$ be a $\omega_{1}$-closed forcing. It is well known that
  such forcings cannot destroy Suslin trees. Therefore we only prove the
  first part.

  Let $\theta$ be a large enough regular cardinal $N \prec H(\theta)$ be
  a countable elementary submodel containing $T$ and $\P$.
  Fix $p \in N \cap \P$. Let $\delta := N \cap \omega_{1}$. We
  need to show that there is an $(N,\P,E_{T})$-generic extension of
  $p$. Let $\tup{D_{n} : n < \omega}$ enumerate all the dense subsets
  of $\P$ which are in $N$, and let $\{(\dt{A}_{n},b_{n},t_{n}) : n < \omega\}$ list
  all the tuples of the form $(\dt{A},b,t)$ where $\dt{A}$ is a $\P$-name
  in $N$,  $b$ is an $(N,E_{T})$-generic
  element of $[T_{\delta}]^{2}$ and $t$ is a type. We recursively
  construct a decreasing sequence $\tup{p_{n} : n < \omega}$ of conditions
  in $N$.

  Let $p = p_{0}$. Assume that for $n < \omega$, $p_{n}$ has been defined
  and it is in $N$. We construct $p_{n+1}$. If $n$ is even, using elementarity
  pick $p_{n+1}$ to be any extension of $p_{n}$ such that $p_{n+1} \in D_{n} \cap N$.
  Now assume $n$ is odd. Let $B_{n}$ be the set of $c \in [T]^{2}$ such that
  there is an $r \le p_{n}$  such that
  $r \forces_\P$ ``$c \in \dt{A}_{n} \seb [T]^{2}$''.
  Notice that $B_{n} \in N$ by elementarity. If $b_{n} \notin B_{n},$ then let
  $p_{n+1} := p_{n}$. Otherwise, since $b_{n}$ is $(N,E_{T})$-generic,
  there is $a \in N \cap B_{n}$ such that $\stp(a,b) = t_{n}$.
  By elementarity there is $p_{n+1} \le p$ in $N$ such that
  $p_{n+1} \forces$``$\dt{a} \in \dt{A}_{n}$''.

  Finally, let $q$ be a lower bound of $\tup{p_{n} : n < \omega}$. One easily sees
  that $q$ is $(N,\P,E_{T})$-generic. Since $q \le p$, this finishes the proof.
\end{proof}

\section{Forcing the Open Graph Axiom}\label{sec:forcing-OGA}

 In this section, we prove that for any open graph $G$ that is not countably chromatic, there exists a forcing notion
$\P_G$ which is $E_S$-proper and $S$-preserving, and adds an uncountable complete subgraph of $G$.
This is achieved by showing that a variant of Todorcevic's forcing \cite{Todorcevic2011} ---specifically, the one introduced in our recent preprint \cite{MP2025}--- is  $E_S$-proper and preserves  the Suslin line $S$.

For the rest of the section, fix an open graph  $G = (X, E)$ on a subset of the reals that is not countably chromatic and weakly bi-entangled lexicographically ordered Suslin line $(S,<_S,<_{\lex})$.
Let $\mathcal{I}$ be the (proper) $\sigma$-ideal consisting of all subsets $Y$ of $X$ which are countably chromatic. Furthermore, by removing a relatively open subset of $X$, we may assume that no nonempty  open subset of $X$ belongs to $\mathcal I$. Let $\mathcal{H}:=\mathcal{I}^+$ denote the corresponding co-ideal.
\begin{definition}
    For each $n>0$, the Fubini power $\mathcal{H}^n$ is the co-ideal on $X^n$ defined recursively by $\mathcal{H}^1:= \mathcal{H}$ and for $n > 1$,
$$\mathcal{H}^n=\{W \subseteq X^n :\{\overline{x} \in X^{n-1} :W_{\overline{x}}\in \mathcal{H}\}\in \mathcal{H}^{n-1}\},$$ where for $W \subseteq X^n$ and $\overline{x} \in X^{n-1}$,
$$W_{\overline{x}}= \{ y \in X : \overline{x}^\frown y \in W \}.$$
\end{definition}


 The following Lemma is proved in \cite[Lemma 5.1]{Todorcevic2011}.
\begin{lemma}\label{lem:boundaryisincoideal}
    Let $n$ be a positive integer and let $W \in \mathcal{H}^n$ and let
$$
\partial W := \{\overline{x}\in W : (\forall \epsilon > 0)(\exists \overline{y}\in W )(\forall i < n)[(x_i,y_i) \in E \wedge |x_i-y_i| < \epsilon]\}
$$
Then $W \setminus \partial W \notin \mathcal{H}^n$.
\end{lemma}
We need one more definition before we introduce our forcing notion.
\begin{definition}

Let $\theta$ be a regular cardinal such that $X, G, \mathcal{I}$ and $(S,<_S,<_{\lex})$ belong to $H(\theta)$ and fix $<_w$ a well-order of $H(\theta)$.
\begin{enumerate}
    \item[(1)]  Let $X \in H(\theta)$, by $\mathcal{SK}(X)$ we denote the Skolem closure of $X$ (where the set of Skolem functions is defined using the  well-order $<_w$ of $H(\theta)$).
\item[(2)]  If $N \prec H(\theta)$ is countable, by $N^+$ we denote $\mathcal{SK}(N \cup \{N\})$.
\end{enumerate}
\end{definition}
The idea of using successors of models in side conditions is a quite natural way to obtain $(N,\P,E_S)$-genericity.
\begin{definition}
Let $\P_G$ be the collection of all pairs $p=(\mathcal{N}_p,f_p)$ satisfying the following conditions:
\begin{enumerate}[{\rm (1)}]
    \item $\mathcal{N}_p=\{N_0,\dots,N_{m}\}$ has the following properties:
    \begin{enumerate}[{\rm (a)}]
        \item For each $i\leq m$, $N_i \prec(H(\mathfrak{c}^+),\in,<_w)$ (countable)  containing $X, G, \mathcal{I}$ and $(S,<_S,<_{\lex})$ as elements.

    \item $N_i\in N_i^+\in N_{i+1}$ for $i<m$.
    \end{enumerate}
\item $f_p:\mathcal{N}_p\to X$ such that
\begin{enumerate}[{\rm (a)}]
    \item For each $i\leq m,$ $f_p(N_i)\in N_{i+1}\setminus N_i^+$ (where $N_{n+1}=X)$.
    \item For each $i\leq m$, $f_p(N_i)\notin \bigcup (\mathcal{I}\cap N_i^+)$.
    \item For each $i<j\leq m$, $(f_p(N_i),f_p(N_j))\in E.$
\end{enumerate}
\end{enumerate}
Let $q\leq p$ if $f_p\subseteq f_q$.
\end{definition}
The next Lemma is straightforward and appears implicit in \cite{Todorcevic2011}. We just write it to have it for future reference.
\begin{lemma}\label{lem:separatedsequenceisincoideal}
    Let $\{N_0,\dots,N_{m}\}$ be an increasing $\in$-chain of countable elementary submodels of $H(\mathfrak{c}^+)$ containing all relevant parameters and let $(x_0,\dots,x_{m})\in X^{m+1}$ be such that $x_i\in N_{i+1}\setminus N_i$ (where $N_{m+1}=X)$ and $x_i\notin \bigcup (N_i\cap \mathcal{I})$ for $i\leq m.$ If $W\subseteq X^{m+1}$ is such that $W\in N_0 $ and $(x_0,\dots,x_m)\in W$, then $W\in \mathcal{H}^{m+1}.$
\end{lemma}

\begin{theorem}
    The forcing notion $\P_G$ is $S$-preserving and $E_S$-proper.
\end{theorem}

\begin{proof}
    We shall prove that $\P_G$ is an $S$-preserving $E_S$-proper poset which adds an uncountable complete subgraph $Y$  of $X$ .

     Let $\overline{M}\prec H((2^{\mathfrak{c}})^+)$ countable which contains all relevant parameters. Consider $p\in \overline{M}\cap \P_G$. We will find $q\leq p$ which is $(M,\P_G,E_S)$-generic. Set  $M:=\overline{M}\cap H(\mathfrak{c}^+)$ and $\delta=M\cap\al$. Define $q=(\mathcal{N}_p\cup \{M\},f_p\cup\{(M,f_{q}(M))\})$ where  $f_q(M)$ is any element of $X\setminus \bigcup (\mathcal{I}\cap M^+)$ and such that $(f_q(M),f_p(N))\in E$ for any $N\in\mathcal{N}_p$. To see that such a $f_q(M)$ exists. Let $N^*$ be the maximum element of $\mathcal{N}_p$. Pick an open rational interval $U$  such that $f_p(N^*)\in U$ and $(f_p(N),x)\in E$ for all $N\in\mathcal{N}_p\cap N^*$ and $x\in U$.
Let $Z$ be the set of all elements of $x\in U$ such that for any $y\in U$  $(x,y)\notin E$. Notice that $Z\in N^*\cap \mathcal{I}$. It follows, from clause 2(b) of the definition of the forcing, that $f_p(N^*)\notin Z$. Pick $y\in U$ such that $(y,f_p(N^*))\in E$. Since $E$ is open there is a rational interval $V$ such that $y\in V\subseteq U$ and $V\times \{x\}\subseteq E.$ Let $f_q(M)$ be any element in $V\setminus (\bigcup (M^+\cap \mathcal{I})\cup M^+)$ (this is possible since all nonempty subsets of $X$ belong to  $\mathcal{H})$.

We claim that $q$ is a $(M,\P_G,E_S)$-generic. To see this,  fix a dense open set $D\in M$, $b\in [S_\delta]^2 $ and $\dot A\in M$ a $\P_G$-name for a subset of $[S]^2$ and a type $t$. Let $r\leq q$. By extending further if necessary, we may assume that $r \in D$, $r$ decides whether or not $b \in\dot A$. If $r$ forces that $b \notin\dot A$, then replace $\dot A$ in what follows by the canonical $\P_G$-name for $[S]^2$. Thus, we may assume without loss of generality that $r$ forces that $b \in \dot A$.

We need to find $s\in D\cap M$ and $a\in [S]^2\cap M$ such that $s$ is compatible with $r$, $\stp(a,b)=t$ and $s\Vdash`` a\in\dot A "$.  Let $r_M:=(\mathcal{N}_r\cap M, f_r\restriction\mathcal{N}_r\cap M).$ It follows that $r_M\in \P_G\cap M$ and $r$ is an end-extension of $r_M$. Let $m=|\mathcal{N}_r\setminus M|$ and $M=N_0,\dots,N_{m-1}$ denote the increasing enumeration of  $\mathcal{N}_r\setminus M$. Let $W$ be the set of all $m$-tuples $(x_0,\dots,x_{m-1}) $ such that there is an end extension $r'$ of $r_M$ in $D$ such that
\begin{enumerate}[{\rm(a)}]
    \item
    $|\mathcal{N}_{r'}\setminus\mathcal{N}_{r_M}|=m$
    \item if $K_0,\dots,K_{m-1}$ is the increasing enumeration of $\mathcal{N}_{r'}\setminus\mathcal{N}_{r_M}$ then $f_{r'}(K_i)=x_i$ for $i<m$.
    \item $r'\Vdash ``b\in \dot A"$.
\end{enumerate}
Using that $E$ is open, fix basic open rational intervals $U_i (i<m)$ such that $f_r(N_i)\in U_i$ and $U_i\times U_j\subseteq E$ for $i\ne j \in m.$
Here is the key point of the definition of the forcing, since $a\in N_0^+$, then $W\in N_0^+$ and $(f_r(N_0),\dots,f_r(N_{m-1}))\in W$. Thus it follows, from Lemma \ref{lem:boundaryisincoideal} and Lemma \ref{lem:separatedsequenceisincoideal}, that $\partial W\in \mathcal{H}^m $ and $(f_r(N_0),\dots,f_r(N_{m-1}))\in \partial W.$ Pick $\varepsilon>0$ such that $B_{\varepsilon}(f_r(N_i))\subseteq U_i$ for all $i<m$. From the definition of $\partial W$ we infer that there is a tuple $(x_0,\dots,x_{m-1})\in W$ such that $(f_r(N_i),x_i)\in E$ and $|f_r(N_i)-x_i|<\varepsilon$ for all $i<m$. Using, once more,  that $E$ is open find open rational intervals  $V_i (i<m)$ such that $x_i\in V_i\subseteq U_i$ and $V_i\times \{f_r(N_i)\}\subseteq E$ for $i<m$. Let $r^*$ be  a witness that  $(x_0,\dots,x_{m-1})\in W$ (we do not claim that $r^*$ is compatible with $r$).

Let $B$ denote the set of pair of nodes $c\in [S]^2$ such that there is an $s\in D$ such that $s\Vdash`` c\in \dot A"$
and $(f_s(K_0),\dots,f_s(K_{m-1}))\in V_1\times\dots\times V_{m-1}$. Notice that $r^*$  witness that $b\in B$. Since $b$ is $(M,E_S)$-generic, there exists $a\in B\cap M$ such that  $\stp(a,b)=t$.  By elementarity, we can find  $s$ in $M$ which is a witness that $a\in B$. We claim that $s$ and $r$ are compatible. Let $u:=s\cup r$. We will show that $u\in \P_G$. Then clearly $u\leq r,s$. It is easy to check that $u$ satisfies all properties for being in $\P$ except perhaps clause (2)(c). To see this, it is sufficient to show that for any $i,j<m,$ $(f_r(N_i),f_s(K_j))\in E.$ We proceed by cases. On one hand, if $i\ne j$, then $(f_r(N_i),f_s(K_j))\in U_i\times U_j\subseteq E$. On the other hand, if $i=j$, then $(f_r(N_i),f_s(K_j))\subseteq \{f_r(N_i)\}\times V_i\subseteq E$. This completes the proof that $q$ is $(M,\P,E_S)$-generic.

The proof that $\P_G$ preserves $S$ follows exactly as the proof in \cite{Todorcevic2011}, except replacing the construction of the $(M,\P_G)$-generic condition with the proof above. Also, note that while in \cite{Todorcevic2011} the Suslin tree is assumed to be coherent, this property is not used in the proof.

Finally, observe that for each $\alpha<\al,$ the set $D_\alpha:=\{p: \alpha\in \bigcup {\rm dom}(f_p)\} $ is dense-open. Thus, if $G$ is any generic filter intersecting the above dense sets, then  the  set $Y:={\rm ran}(\bigcup_{p\in G}  f_p)$ is an uncountable complete subgraph. This concludes the proof of the Theorem.
\end{proof}

\section{Preservation  under countable support iterations.}\label{sec:iteration}

From the previous section, we know that any particular instance of OGA can be forced while preserving that $(S,<_S,<_{\lex})$ remains $2$-entangled, by using an $E_S$-proper and $S$-preserving forcing.  Our current goal is to prove that countable support iterations of $E_S$-proper forcing remain $E_S$-proper.  In order to do this, we shall use an enhanced version of  countable support iterations developed  by Schlindwein.

The reader may wonder what distinguishes Schlindwein's approach from more classical frameworks. A key difference lies in how fullness is managed in countable support iterations. Classically, it suffices that the set of conditions is full at each coordinate. In Schlindwein's framework, however, this fullness must be witnessed by conditions that cohere in a specific, well-behaved manner with the earlier posets in the iteration. Roughly speaking, while classical constructions may use any saturation of the posets, Schlindwein makes a careful selection to ensure a particularly ``nice$"$ saturation. This technical distinction is crucial for the ensuing arguments.

For convenience of the reader, we will review the work of Schlindwein.

Let $\P$ be a forcing poset and let $\dQ$ be a $\P$-name for a forcing poset. We define $\P*\dQ=\{(p,\dt{q}): p\in \P \wedge p\forces_\P``\dt{q}\in \dQ " \wedge \P{\rm -rank}(\dt{q})\leq\P{\rm -rank}(\dQ)\}$ where $\P$-rank is recursively defined for all $\P$-names as follows: $$\P{\rm -rank}(x)=\sup\{\P{\rm -rank}(y)+1: \exists p\in \P, (y,p)\in x\}. $$

The following  Definition is \cite[Definition 69]{Schlindwein94}.
\begin{definition}
    Let $\gamma$ be an ordinal. A $\gamma$-stage countable support iteration  is a pair of sequences $\langle \P_\xi,\dt{\mathbb{Q}}_\eta:\xi\leq\gamma, \eta< \gamma\rangle  $ so that:
    \begin{enumerate}[{\rm (1)}]
        \item Each $\P_\xi$ is a forcing poset.
        \item All conditions in $\P_\xi$ are sequences of length $\xi.$
        \item $\dQ_\xi$ is an ordered triple $\langle \dQ_\xi, \dt{\leq}_{\xi},\dt{\mathbbm{1}}_{\xi}\rangle$ and it is forced by all conditions in $\P_\xi$ that ``$\dt{\leq}_\xi $ is a  partial order on $\dQ_\xi$ with largest element $\dt{\mathbbm{1}}_{\xi}"$.
        \item $\P_0$ is the trivial poset $=\{0\}.$
        \item Conditions in $\P_{\xi+1}$ are sequences $p$ such that:
        \begin{enumerate}[{\rm (a)}]
            \item $p\restriction\xi\in\P_\xi$, and
            \item $p\restriction \xi\forces_{\P_\xi}``p(\xi)\in \dQ_\xi"$  and $\P_\xi{\rm -rank}(p(\xi))\leq\P_\xi{\rm -rank}(\dQ_\xi).$

            The ordering on conditions is $p^*\leq p$ iff $p^*\upharpoonright\xi\leq p\restriction\xi $ and $$p^*\restriction\xi\forces_{\P_\xi} ``p^*(\xi)\ \dt{\leq}_\xi \  p(\xi)".$$
        \end{enumerate}
        \item For limit $\eta,$ $\P_\eta$ consists of sequences $p$ so that:
        \begin{enumerate}[{\rm (a)}]
        \item $p\restriction{\xi}\in \P_\xi$ for all $\xi<\eta,$ and
        \item The support $\supt(p) :=\{\xi<\eta: p\res\xi\not\forces_{\P_\xi}``p(\xi)=\dt{\mathbbm{1}}_{\xi}"\}$ is countable.

        The ordering is given by $p^*\leq p$ iff $\forall\xi<\eta, p^*\restriction \xi\leq p\restriction\xi$.
        \end{enumerate}
        \item $\dt{\mathbbm{1}}_{\gamma}\restriction\xi=\dt{\mathbbm{1}}_{\xi}$ and $\P_\xi$-rank$(\dt{\mathbbm{1}}_{\gamma}(\xi))\leq \P_\xi$-rank$(\dQ_\xi)$.
        \item For any $\xi\leq\eta\leq \gamma,$ $q\in \P_\eta, p\in \P_\xi$. If $q\restriction \xi\leq p$ then $q\restriction \xi\cup p\restriction[\xi,\eta)\in \P_\eta.$
    \end{enumerate}
\end{definition}

\begin{definition}\label{schlindwein-maindef}
   Let  $\langle \P_\xi,\dt{\mathbb{Q}}_\eta:\xi\leq\gamma, \eta< \gamma\rangle  $ be a countable support iteration. For each $\xi\leq \gamma $, we recursively define $I_\gamma^\xi(\dt{a})$  for every $\P_\gamma$-name $\dt{a}$ as follows: Suppose $I_\eta^\xi(\dt{b})$  has been defined for all $\P_\eta$-names $\dt{b}$, for all $\xi<\eta<\gamma$  and that $I_\gamma^\xi(\tau)$  has been defined for all $\P_\gamma$-names $\tau$ such that $\P_\gamma$-rank$(\tau) < \P_\gamma$-rank$(\dt{a})$. For $p\in \P_\gamma$, we define  $p\restriction [\xi,\gamma)$  to be a $\P_\xi$-name which is forced to be a function with domain $[\xi,\gamma)$, and such that whenever $\xi<\eta<\gamma$  we have $\dt{\mathbbm{1}}_{\xi}\forces_{\P_\xi} ``p\restriction
[\xi,\gamma)(\eta) = I_\eta^\xi(p(\eta))"$.
For $x,y$ $\P_\xi$-names, let op$(x,y)$ denote the $\P_\xi$-name forced to be the order pair of $x$ and $y$. Finally, set $I_\xi^\gamma(\dt{a})$ to be equal to
$$
\{\langle {\rm op}(\tau,p\res [\xi,\gamma)),p\res \xi\rangle: p\forces_{\P_\gamma}``\tau\in\dt{a}"\ \wedge \P_\gamma{\rm-rank}(\tau)<\P_\gamma{\rm -rank}(\dt{a})\}.
$$
\end{definition}
\begin{rmk}
The $\P_\xi$-name $p\restriction [\xi,\gamma)$ will always refer to one given in Definition \ref{schlindwein-maindef}.
\end{rmk}
\begin{definition}
    Let $\langle \P_\xi,\dt{\mathbb{Q}}_\eta:\xi\leq\gamma, \eta< \gamma\rangle  $ be an iteration. For each $\xi<\gamma,$ define $\dt{\P}_{\xi,\gamma}$ to be the $\P_\gamma$-name characterized by: $p\forces_{\P_{\xi}}`` \dt{q}\in \dt{\P}_{\xi,\gamma}"$ iff for every $p'\leq p $ there is an $q^*\in \P_\gamma$ such that $q^*\restriction{\xi} \leq p'$ and $q^*\restriction {\xi}\forces_{\P_{\xi}}`` q^*\restriction[\xi,\gamma)=\dt{q}"$ where   $q^*\restriction[\xi,\gamma)$ is as in Definition \ref{schlindwein-maindef}.
\end{definition}
The next Lemma appears as \cite[Lemma 72]{Schlindwein94}.
\begin{lemma}\label{lem:schlindweinmain}
    Let  $\langle \P_\xi,\dt{\mathbb{Q}}_\eta:\xi\leq\gamma, \eta< \gamma\rangle  $ be a  countable support iteration. Then the following conditions are satisfied.
    \begin{enumerate}[{\rm (1)}]
       \item $\dt{\mathbbm{1}}_{\xi}\forces_{\P_\xi} ``\dt{\P}_{\xi,\gamma}$ is a poset  $"$.
       \item If $p\forces_{\P_\xi} ``\dt{b}$ is a $\dt{\P}_{\xi,\gamma}$-name $"$, then there exists a $\P_\gamma$-name $\dt{a}$ such that $p\forces_{\P_\xi} ``\dt{b}=I_\gamma^\xi(\dt{a})"$ and $\P_\gamma$-rank$(\dt{a})\leq \P_\xi$-rank$(\dt{b})$.
       \item If $\varphi(v_1,\dots,v_n)$ is a first-order formula in the language of  set theory with free variables among $v_1,\dots,v_n$ and $p\in\P_\gamma$, and  $\dt{a}_1,\dots,\dt{a}_n$ are $\P_\gamma$-names, then
        $$ p\restriction\xi\forces_{\P_\xi}``p\restriction [\xi,\gamma)\forces_{\dt{\P}_{\xi,\gamma}} \varphi(I_\gamma^\xi(\dt{a}_1),\dots,I_\gamma^\xi(\dt{a}_n))" $$
        if and only if

        $$ p\forces_{\P_\gamma}``\varphi(\dt{a}_1,\dots,\dt{a}_n)". $$
    \end{enumerate}
\end{lemma}
The following definition is somewhat unusual, but it provides extra flexibility by  allowing us to use names outside of $N.$
\begin{definition}\label{def-N[G]}
Let $\P$ be a forcing notion.
    Let $N[\dt{G}]$ denote the $\P$-name characterized as follows: for any condition $p\in \P$ and any $\P$-name $\dt{a}$ we have that  $ p\forces_\P ``\dt{a}\in N[\dt{G}_\P]"$ if and only if for every $p_1\leq p$ there exists $p_2\leq p_1$ and $\dt{b}\in N$ such that $p_2\forces_\P \dt{a}=\dt{b}.$
\end{definition}

The following Lemma appears  in \cite[Lemma 13]{Schlindwein94}.
\begin{lemma}\label{schlindwein13}
    Let  $\langle \P_\xi,\dt{\mathbb{Q}}_\eta:\xi\leq\gamma, \eta< \gamma\rangle  $ be a  countable support iteration. Let $\theta$ be a large enough regular cardinal and $N$ a countable elementary submodel of $H(\theta)$ containing $\gamma, \P_\gamma$ and $\xi\in N\cap \gamma.$ If $q$ is $(N,\P_\xi)$-generic and $q\forces_{\P_\xi}``\dt{p}\in \dt{\P}_{\xi,\gamma}\cap N[G_\xi]",$ then there exists $q^*\in \P_\gamma$ such that $q^*\restriction\xi=q$ and $q\forces_{\P_\xi}``q^*\restriction [\xi,\gamma)=\dt{p}"$ and $\supt(q^*)\subseteq \xi\cup N.$
\end{lemma}

The following Lemma was proven in \cite[Lemma 14]{Schlindwein94}.
\begin{lemma}\label{schlindwein14}
    Suppose that $\langle \P_\xi,\dt{\mathbb{Q}}_\eta:\xi\leq\gamma, \eta< \gamma\rangle  $, $N $ are as in Lemma \ref{schlindwein13} and $\xi<\eta<\gamma $ all belong to $N$ and $p\in \P_\gamma$ is a condition such that $p\restriction\xi\forces_{\P_\xi}``p\restriction[\xi,\gamma)\in N[\dt{G}_\xi]",$ then $p\restriction\eta\forces_{\P_\eta}``p\restriction[\eta,\gamma)\in N[\dt{G}_\eta]"$.
\end{lemma}

The following Lemma is proven in \cite[Lemma 2.5]{Abraham2010}
\begin{lemma}\label{lem:2genericsaregenericiniteration}
    Let $\P$ be a  forcing poset and $\dt{\mathbb{Q}}$ a $\P$-name for a forcing poset. Let $\theta$ be a large enough regular cardinal and $N$ a countable elementary submodel of the structure $(H(\theta),\in,\P*\dt{\mathbb{Q}}) $ then $(p,\dt{q})$ is $(N,\P*\dt{\mathbb{Q}})$-generic iff
    $$p \ {\rm is}\ (N,\P){\rm-generic}  $$
    and
    $$p\forces_\P `` \dt{q} \ {\rm is}\ (N[\dt{G}_\P],\dt{\mathbb{Q}}){\rm-generic}".$$

\end{lemma}\label{lem:b-n-generic-implica-ng-generico}
The next Lemma will be used to prove that the two step iteration of $E_S$-proper forcings is $E_S$-proper.
\begin{lemma}
    Let $\P$ be a forcing poset and let $(S,<_S,<_{\lex})$ be a weakly bi-entangled lexicographically ordered Suslin line. Let $\theta$ be a large enough regular cardinal and let $N$ be a countable elementary submodel of $H(\theta)$ containing $\P$ and $S$ as elements. Suppose $q$ is $(N,\P,E_S)${\rm -generic}. If $b\in[S_{N\cap\al}]^2$ is $(N,E_S)${\rm-generic}, then $q\forces_\P ``b \ {\rm \ is}\ (N[\dt{G}],E_S){\rm-generic}". $
\end{lemma}
\begin{proof}
Let $\dt{A}$ be a $\P$-name for a subset of $[S]^2$  such that $q\forces_\P`` \dt{A}\in N[\dt{G}]" $ (we are not assuming that $\dt{A}\in N$ ). Fix a type $t:2\to\{<,>\}$. Suppose that $r$ is  an extension  of $q$ such that $r\forces_\P``b\in\dt{A}"$. By Definition \ref{def-N[G]}, there exists $r'\leq r$ and a $\P$-name $\dt{B}$ in $N$ such that $r'\forces_\P`` \dt{A}=\dt{B}$. Using our assumptions on $q$ and $b$, we can find a condition $s\leq r'$ and a pair of nodes $a\in M\cap [S]^2$ such that $\stp(a,b)=t$ and $s\forces_\P``a\in\dt{A}".$ This concludes the proof of the Lemma.
\end{proof}

The next Lemma plays a similar role, as Lemma \ref{lem:2genericsaregenericiniteration}, for $(M,\P,E_S)$-generic conditions.
\begin{lemma}\label{2stepiteration} Let $(S,<_S,<_{\lex})$ a weakly bi-entangled Suslin tree.
    Let $\P$ be a $E_S$-proper forcing poset  and $\dt{\mathbb{Q}}$ a $\P$-name for an  $E_S$-proper forcing poset. Let $\theta$ be a large enough regular cardinal and $N$ a countable elementary submodel of the structure $(H(\theta), \in,S,\P*\dt{\mathbb{Q}})$. If $p$ is  $(N,\P,E_S)$-generic  and $p\forces_\P`` \dt{q}\ {\rm is}\ (N[\dt{G}_\P],\dQ,E_S){\rm -generic}" $ then $(p,\dt{q})$ is $(N,\P*\dQ,E_S)$-generic.

\end{lemma}
\begin{proof}
To prove that $(p,\dt{q})$ is $(N,\P*\dQ,E_S)$-generic, let $b\in [S_{N\cap\al}]^2$, $\dt{A}\in N$  a $\P*\dQ$-name for a subset of $[S]^2$ and  a type $t:2\to\{<,>\}$. Suppose that $(p_1,\dt{q}_1)\leq (p,\dt{q})$ and $(p_1,\dt{q}_1)\forces`` b\in\dt{A}".$ Let $\dt{B}:=I_2^1(\dt{A})$ (where we view $\P*\dQ$ as a 2-stage iteration). Notice that $\dt{B}\in N$ since is definable from $\P*\dQ$ and $\dt{A}$.

Our goal is to find $a\in[S]^2\cap N$ and $(p_2,\dt{q}_2)\leq (p_1,\dt{q}_1)$ such that $\stp(a,b)=t$ and $(p_2,\dt{q}_2)\forces`` a\in \dt{B}".$ To do this, let $G$ be any $(V,\P)$-generic filter containing $p_1$. Since $p\in G$ is $(N,\P,E_S)$-generic, then it follows from Lemma \ref{lem:b-n-generic-implica-ng-generico}  that $b$ is $(N[G],E_S)$-generic. By Lemma \ref{lem:schlindweinmain}(3), we have that  $p_1\forces_\P``\dt{q}_1\forces_{\dQ} b\in \dt{B}" $. Hence,   $\dt{q}_1[G]\forces ``b\in \dt{B}[G] "$ since $\dt{q}_1[G]$ is $(N[G],\dQ[G],E_S)$-generic and $\dt{B}[G]\in N[G]$ is a $\dQ[G]$-name for a subset of $[S]^2$, then we can find $a\in N[G]\cap [S]^2 $ and $q_2\leq \dt{q}_1[G]$ such that $\stp(a,b)=t$ and $q_2\forces`` a\in \dt{B}[G]".$ Since $p$ is $(N,\P)$-generic, we have that $N[G]\cap V=N\cap V$ and thus $a\in N$. Thus,

$$
H(\theta)[G]\models \exists q_2\leq q_1[G], a\in N (\stp(a,b)=t \land q_2\forces_{\dQ}``a\in \dt{B}[G]").
$$
Since $G$ was an arbitrary generic filter containing $p_1$, then we obtain that
$$
p_1\forces_\P ``\exists x\leq \dt{q}_1, a\in N (\stp(a,b)=t \land x\forces_{\dQ}``a\in \dt{B}[\dt{G}_\P])".
$$
By existential completeness there is a $\P$-name $\dt{q}_2$ such that
$$
p_1\forces_\P ``\dt{q}_2\leq \dt{q}_1, a\in N (\stp(a,b)=t \land \dt{q}_2\forces_{\dQ}``a\in \dt{B}[\dt{G}_\P])".
$$
 Using Lemma \ref{lem:schlindweinmain}(3), we obtain that $(p_2,\dt{q}_2)\leq (p_1,\dt{q}_1) $
and $(p_2,\dt{q}_2)\forces_{\P*\dQ}`` a\in \dt{A}$.

\end{proof}

The next Lemma was proven in \cite[Lemma 16]{Schlindwein94}.
\begin{lemma}\label{schlindwein16}
    Let $\P$ be a  forcing poset and $\dt{\mathbb{Q}}$ a $\P$-name for a forcing poset. Let $\theta$ be a large enough regular cardinal and $N$ a countable elementary submodel of the structure $(H(\theta),\in,\P*\dt{\mathbb{Q}}) .$
    Suppose $p$ is $(N,\P)$-generic, $p\forces_\P``\dt{q}\in \dQ\cap N[\dt{G}_\P]"$ and and $\tau\in N $ is a $\P*\dQ$-name such that $\forces_{\P*\dQ}``\tau\in {\rm ON}".$ Then there is a $\dt{q}_*$ such that $p\forces_\P``\dt{q}_*\leq \dt{q}\ \wedge \  \dt{q}_*\in N[\dt{G}_\P]"$ and $(p,\dt{q}_*)\forces_{\P*\dQ}`` \tau\in N"$.
\end{lemma}
Our iteration result rests on the following key lemma.
\begin{lemma}\label{key:limititeration}
    Let $S$ be a weakly bi-entangled Suslin line. Let $\P$ be an $E_S$-proper  forcing poset and $\dt{\mathbb{Q}}$ a $\P$-name for a  $E_S$-proper forcing poset. Let $\theta$ be a large enough regular cardinal and $N$ a countable elementary submodel of the structure $(H(\theta),\in,\P*\dt{\mathbb{Q}},S) .$
    Suppose $p$ is $(N,\P,E_S)$-generic, $p\forces_\P``\dt{q}\in \dQ\cap N[\dt{G}_\P]"$  and $\dt{A}\in N $ is a $\P*\dQ$-name for a subset of $[S]^2$, $t$ is a type and $b\in[S_\delta]^2$ where $\delta:=N\cap\al$.  Then there is an $\dt{s}$ such that $p\forces_\P``\dt{s}\leq \dt{q}\ \wedge  \dt{s}\in N[\dt{G}_\P]"$ and $(p,\dt{s})\forces_{\P*\dQ}``$ if $b\in\dt{A}$ then there is an $a\in N\cap \dt{A}$ such that $\stp(a,b)=t"$.
\end{lemma}
\begin{proof}
Let $\dt{B}:=I_2^1(\dt{A})$ (where we view  $\P*\dQ$ as a 2-stage iteration).
Let $D$ denote the set of all conditions $p'\leq p$ such that either:
\begin{enumerate}[{\rm (a)}]
    \item $p'\forces_\P ``\dt{q}\forces_{\dQ} b\not \in \dt{B}"$ or
    \item there are $\dt{r}$ and $a\in [S\res \delta]^2$ such that $\stp(a,b)=t$ and $$p'\forces_\P ``\dt{r}\leq \dt{q}\ \wedge \ \dt{r}\in N[\dt{G}_\P]\wedge \dt{r}\forces_{\dQ} a\in \dt{B}". $$
\end{enumerate}
\begin{claim}
    We claim that $D$ is dense below $p$.
\end{claim}
\begin{proof}
To see this,  fix $p^*\leq p$ and suppose, without loss of generality, that $p^*\forces_\P ``\dt{q}\not\forces_{\dQ}b\notin \dt{B}".$ By further extending $p^*$, if necessary, we can find $\dt{q}_*\in N$ such that $p^*\forces_\P ``\dt{q}_*=\dt{q}".$
    Let $G$ be any $(V,\P)$-generic filter containing $p^*$. In $V[G]$, we have that $\dt{q}_*[G]\in N[G]$ and since $p$ is $(N,\P,E_S)$-generic we have that  $b$ is  $(N[G],E_S)$-generic.  Consider the set $X:=\{b\in [S]^2: \dt{q}_*[G]\not\forces_{\dQ[G]}b\notin \dt{B}[G]\}.$ Notice that $X\in N[G]$ as is definable from the parameters $S,\dt{B}$ and $\dt{q}_*$ which are in  $N$. Using that
$b$ is  $(N[G],E_S)$-generic, we infer that there is an $a\in[S\res \delta]^2$ such that $a\in X$ and $\stp(a,b)=t$. Since $a\in X$, we can find $r\leq \dt{q}_*[G]$ such that $r\forces_{\dQ[G]}`` a\in \dt{B}[G]".$

Since $\theta$ is large enough, it follows that
$$H(\theta)[G]\models \exists r\leq \dt{q}_*[G] \left (r\forces_{\dQ[G]}``a\in \dt{B}[G]"\right ) $$
as $N[G]\prec H(\theta)[G]$, then by Tarski-Vaugh there is such an $r\in N[G].$ In other words, there is an $r\in  \dQ[G]$ such that
$$H(\theta)[G]\models r\leq \dt{q}_*[G]\left ( r\in N[G] \ \wedge \ r\forces_{\dQ[G]}``a\in \dt{B}[G]" \right ).$$ Fix a $\P$-name $\dt{r}$ for such an $r$. Thus,
$$
V[G]\models \dt{r}[G]\leq \dt{q}_*[G]\left ( \dt{r}[G]\in N[G] \ \wedge \ \dt{r}[G]\forces_{\dQ[G]}``a\in \dt{B}[G]" \right )
$$
By the Theorem of forcing there is a condition $p'\in G $, which we may, and will assume, is less than $p^*$ such that
$p'\forces_\P ``\dt{r}\leq \dt{q}_* \left (\dt{r}\in N[\dt{G}_\P] \ \wedge \ \dt{r}\forces_{\dQ}``a\in \dt{B}"\right ).$
Hence, $p'\in D$ and therefore $D$ is dense below $p$. This concludes the proof of the Claim.
\end{proof}
Let $I$ be a maximal antichain inside $D$. We now define a function $f$ with domain $I$ as follows. If $p'\in I$ and $p'\forces_\P``\dt{q}\forces_{\dQ}b\notin\dt{B}"$, then let $f(p')=\dt{q}$. Otherwise, we take $f(p')$ to be any condition $\dt{r}$ witnessing that $p'$ satisfies clause (b) in the definition of $D$. By the definition by cases Lemma, there is a $\P$-name $\dt{s}$ such that $p'\forces_\P``\dt{s}=f(p')"$ for all $p'\in I.$

 We shall prove that $\dt{s}$ satisfies the conclusion of the Lemma. First, we claim that  $p\forces_\P``\dt{s}\leq \dt{q}$ and  $\dt{s}\in N[\dt{G}_\P] "$. To see this, fix $p^*\leq p$ and let $p'\in I$ be a condition compatible with $p^*$. By further extending, $p^*$ we may assume that $p^*\leq p'$. Since $p'\forces_\P``f(p')=\dt{r}, \dt{r}\leq \dt{q} $ and $\dt{r}\in N[\dt{G}_\P]",$ then so does $p^*$. Since $p^*$ was arbitrary, it follows that $p$ forces that $\dt{s}\in N[\dt{G}_\P]$.

 Finally, let us show that $(p,\dt{s})\forces_{\P*\dQ}``$ if $b\in\dt{A}$ then there is a $a\in[S\res \delta]^2\cap \dt{A}$ such that $\stp(b,a)=t"$. In order to see this, fix an extension $(p^*,\dt{s}_*)\leq (p,\dt{s})$ such that $(p^*,\dt{s}_*)\forces_{\P*\dQ}``b\in \dt{A}".$ By further extending $p^*$, if necessary, we may assume that $p^*\leq p'$ for some $p'\in I$.
 Using Lemma \ref{lem:schlindweinmain} (3), we obtain that $p^*\forces_\P``\dt{s}_*\forces_{\dQ} b\in \dt{B}".$ Thus, we have that there is an $a\in [S\res \delta]^2$ such that $\stp(a,b)=t$ and $p'\forces_\P``f(p')=\dt{s}\wedge f(p')\forces_{\dQ} a\in \dt{B}"$.  It follows that $p^*\forces_\P ``\dt{s}_*\leq \dt{s}"$. Using Lemma \ref{lem:schlindweinmain} (3). we obtain that $(p^*,\dt{s}_*)\forces_{\P*\dQ}``a\in \dt{A}".$ This concludes the proof of the Lemma.

\end{proof}

We are now ready to prove the main Theorem of the section.
\begin{theorem}\label{main:preservationthm}

    Let $(S,<_s,<_{\lex})$ be a weakly bi-entangled Suslin  line. Let $\langle \P_\xi,\dt{\mathbb{Q}}_\eta:\xi\leq\gamma, \eta< \gamma\rangle  $ be a countable support iteration of $E_S$-proper forcing posets. Let $\theta$ be a large enough regular cardinal. Let $N$ be a countable elementary submodel of $H(\theta)$  which contains $\gamma,\P_\gamma, S.$  For all $\xi\in N\cap \gamma$ and $q\in \P_{\xi}\cap N$ which is $(N,\P_{\xi},E_S)$-generic the following holds. If $\dt{p}_{\xi,\gamma}$ is a $\P_{\xi}$-name such that
    $$q\forces_{\P_{\xi}}`` \dt{p}_{\xi,\gamma}\in \dt{\P}_{\xi,\gamma}  \cap N[\dt{G}_{\xi}]"$$
    where $\dt{G}_{\xi}$ denotes the canonical $\P_{\xi}$-name for the generic filter over $\P_\xi$, then there exists an $(N,\P_\gamma,E_S)$-generic condition $q^*$ such that $supt(q^*)\subseteq \xi\cup N$, $ q^*\restriction_{\xi}=q$ and
    $$ q\forces_{\P_{\xi}}`` q^*\restriction_{[\xi,\gamma)}\leq \dt{p}_{\xi,\gamma}".$$

\end{theorem}

\begin{proof}
    The proof proceeds by induction on $\gamma$. So suppose the statement holds for all $\eta<\gamma$.  Fix $\xi\in\gamma\cap N$,  $q$ and $\dt{p}_{\xi,\gamma}$ as in the hypothesis of the Theorem. Let $\delta:=N\cap\al$. We will build $q^*$ to witness the conclusion.

    First consider the case that $\gamma$ is a successor, let say $\gamma=\eta+1$. Notice that $\eta$ is also in $N$. Find, using Lemma \ref{schlindwein13}, a condition $q'\in \P_\gamma$ such that $q'\restriction\xi=q$, $$q\forces_{\P_\xi}``q'\restriction[\xi,\gamma)=\dt{p}_{\xi,\gamma}" \ {\rm and}\  \supt(q')\subseteq \xi\cup N. $$
    Since $\eta\in N$ and  $q\forces_{\P_{\xi}}`` \dt{p}_{\xi,\gamma}\in \dt{\P}_{\xi,\gamma}  \cap N[\dt{G}_{\xi}]$, it follows that $q\forces_{\P_\xi}``q'\restriction [\xi,\eta)\in \dt{\P}_{\xi,\eta}\cap N[\dt{G}_\xi]".$ By induction hypothesis, there exists  an $(N,\P_{\eta},E_S)$-generic condition $q^+$ such that $q^+\restriction\xi=q$, $q\forces_{\P_\xi}``q^+\restriction[\xi,\eta)\leq q'\restriction[\xi,\eta)"$ and $\supt(q^+)\subseteq \xi\cup N.$
Since $q\forces_{\P_\xi}``q'\restriction[\xi,\gamma)\in N[\dt{G}_\xi]"$ and $\eta\in N$, then it follows, from Lemma \ref{schlindwein14}, that $q'\restriction \eta\forces_{\P_\eta}``q'(\eta)\in \dQ_\eta\cap N[\dt{G}_\eta]".$ Thus $q^+\restriction \eta$ also forces the above statement. Since $q^+$ is $(N,\P_\eta,E_S)$-generic, then $q\forces_{\P_\eta}``N[\dt{G}]\cap \al=\delta.$ Using that $\forces_{\P_\eta}``\dQ_\eta $ is $E_S$-proper$"$ and existential completeness, we can find a condition $q^*\in\P_\gamma $ such that $q^*\restriction \eta=q^+$ and $q^+\forces_{\P_\eta}``q^*(\eta)\leq q'(\eta)$ and $q^*(\eta)$ is $(N[\dt{G}_\eta],\P_\eta,E_S)$-generic$"$.  It follows from Lemma \ref{2stepiteration} that $q^*$ is $(N,\P_\gamma,E_S)$-generic and by construction satisfies the conclusion of the Theorem. This concludes the successor step part of the induction.

At this point the reader may be wondering if we are being overly pedantic in the amount of details but in view of the proof of Lemma 14 in \cite{Schlindwein94} it seems that some extra care is needed.

Next suppose that $\gamma$ is a limit ordinal.  Fix an increasing sequence of ordinals $\langle \gamma_n:n\in\omega\rangle$  cofinal in $N\cap \gamma$ with $\gamma_n\in N$ and
$\gamma_0=\xi.$  Let $\langle \sigma_n:n\in\omega\rangle$ enumerate the set of all  $\P_\gamma$-names $\sigma$  in $N$ such that $\forces_{\P_\gamma}``\sigma\in {\rm ON}" $.  Let $\langle (b_n,\dt{A}_n,t_n):n\in\omega\rangle$ list  all triples  $(b,\dt{A},t)$, where $b\in [S_{\delta}]^2$, $\dt{A}\in N$ is $\P_\gamma$-name for a subset of $[S]^2$  and  $t:2\to\{<,>\}$ is a type.

    We will recursively construct for $n<\omega$ conditions $q_n\in \P_{\gamma_n}$ and $p_n\in \P_\gamma$ such that:
    \begin{enumerate}[{\rm (1)}]
        \item $q_0=q$ and $p_0\restriction\xi=q$ and $q\forces_{\P_\xi}``p_0\restriction[\xi,\gamma)=\dt{p}_{\xi,\gamma}";$
        \item $q_n\in\P_{\gamma_n}$ is $(N,\P_{\gamma_n},E_S)$-generic;
        \item $p_n\in \P_\gamma$ and $p_n\restriction{\gamma_n}=q_n$;
        \item $q_{n+1}\res\gamma_n= q_n$;
        \item $q_{n+1}\leq p_n\restriction\gamma_{n+1}$;
        \item $q_n\forces_{\P_{\gamma_n}}``p_n\restriction{[\gamma_n,\gamma)}\in \dt{\P}_{\gamma_n,\gamma}\cap N[\dt{G}_{\gamma_n}]"$;
        \item $p_{n+1}\leq p_n$;
        \item $p_{n+1}\forces_{\P_{\gamma_n}}``\sigma_n\in \check{N}"$;
        \item $p_{n+1}\forces_{\P_\gamma}``$if $b_n\in\dt{A}_n$ then exists  $ a\in[S\restriction\delta]^2 $ such that $ \stp(a,b)=t$ and $  a\in\dt{A}_n"; $
        \item $\supt(q_n)\subseteq \gamma_0\cup N;$
        \item $ \supt(p_n)\subseteq \gamma_0\cup N$.

        \end{enumerate}

        Let $q_0:=q$ and find, using Lemma \ref{schlindwein13}, a condition $p_0$ such that $p_0\restriction\xi=q$ and $q\forces_{\P_\xi}``p_0\restriction[\xi,\gamma)=\dt{p}_{\xi,\gamma}".$ Thus, $q_0, p_0$ satisfy clause (1).

Assume that $n<\omega$ and $q_n$ and $p_n$  have been defined.

From clause (6) and the fact that $\gamma_{n+1}\in N$, we infer that

\[ q_n\forces_{\P_{\gamma_n}}``p_n\restriction[\gamma_n,\gamma_{n+1})\in \dt{\P}_{\gamma_n,\gamma_{n+1}}\cap N
[\dt{G}_{\gamma_n}]".\]

Thus, applying the induction hypothesis we obtain an $(N,\P_{\gamma_{n+1}},E_S)$-generic condition $q_{n+1}$ satisfying the following: \[ q_{n+1}\restriction\gamma_n=q_n, \  \supt(q_{n+1})\subseteq \gamma_0\cup N \]  and  \[ q_n\forces_{\P_{\gamma_n}}``q_{n+1}\restriction[\gamma_n,\gamma_{n+1})\leq p_n\restriction[\gamma_n,\gamma_{n+1})".\]

This implies that $q_{n+1}\leq p_n\restriction\gamma_{n+1}$ and hence, $q_{n+1}$ satisfies clauses (2), (4), (5) and (10).

    By clauses (3) and (6), $p_n\res \gamma_n\forces_{\P_{\gamma_n}}``p_n\restriction[\gamma_n,\gamma)\in \dt{\P}_{\gamma_n,\gamma}\cap N[\dt{G}_{\gamma_n}]"$ and thus, by Lemma \ref{schlindwein14}, $p_n\res \gamma_{n+1}\forces_{\P_{\gamma_{n+1}}}``p_n\restriction[\gamma_{n+1},\gamma)\in \dt{\P}_{\gamma_{n+1},\gamma}\cap N[\dt{G}_{\gamma_{n+1}}]"$.

Our goal to extend $q_{n+1}$ to a condition $p_{n+1}'\in \P_\gamma$ that ensures that clause (8) holds. To do this, use Lemma \ref{schlindwein16} to find a $\P_{\gamma_{n+1}}$-name $\dt{r}$ such that $$q_{n+1}\forces_{\P_{\gamma_{n+1}}}"\dt{r}\leq p_{n}\res [\gamma_{n+1},\gamma)\ \wedge \ \dt{r}\in N[\dt{G}_{\gamma_{n+1}}]" \ {\rm and}\  (q_{n+1},\dt{r})\forces_{\P_\gamma}``\tau\in N".$$
    Using Lemma \ref{schlindwein13}, find
    $p_{n+1}'\in\P_\gamma$ such that $p_{n+1}'\res\gamma_{n+1}=q_{n+1}$ and $$q_{n+1}\forces_{\P_{\gamma_{n+1}}}``p_{n+1}'\res [\gamma_{n+1},\gamma)=\dt{r}".$$

 Now our goal is to extend $p_{n+1}'$ to a further condition satisfying clause (9) while maintaining  the rest of the clauses. To do this, we apply Lemma \ref{key:limititeration} to the conditions $q_{n+1}$ and to the $\P_{\gamma_{n+1}}$-name $p_n'\res [\gamma_{n+1},\gamma)$ to obtain a $\P_{\gamma_{n+1}}$-name $\dt{s}$ such that $q_{n+1}\forces_{\P_{\gamma_{n+1}}}``\dt{s}\leq p_{n+1}'\res [\gamma_{n+1},\gamma)\ \wedge \ \dt{s}\in N[\dt{G}_{\gamma_{n+1}}]"$ and $(q_{n+1},\dt{s})\forces_{\P_\gamma}``$if $b_n\in\dt{A}_n$ then there exists $a\in[S\res \delta]^2\cap\dt{A}_n$ such that $\stp(a,b)=t".$ Next, using Lemma \ref{schlindwein13}, we can find $p_{n+1}$ such that $p_{n+1}\res \gamma_{n+1}=q_{n+1}$ and $q_{n+1}\forces_{\P_{\gamma_{n+1}}}``p_{n+1}\res [\gamma_{n+1},\gamma)=\dt{s}".$ This completes the recursive construction.
 It is clear from the construction that $q_n$ and $p_n$ satisfies all clauses.

Let $q^*=\bigcup_{n\in \omega} {q_n}^\frown \mathbbm{1}_{\P_\gamma}\res [\sup(\gamma\cap N),\gamma)$. It follows from clauses (4) and (10) that $supt(q^*)\subseteq \supt(q_0)\cup N $ and thus $q^*\in \P_\gamma$.
We are left to verify that $q^*$ is an  $(N,\P_\gamma,E_S)$-generic condition. First we prove that is $(N,\P_\gamma)$-generic. To see this, notice that, by clause (5) and an easy induction argument, $q^*\leq p_n$ for all $n\in\omega$. Hence, from clause (8), we obtain that $N\cap {\rm ON}=N[G_\gamma]\cap {\rm ON}$ for any $(V,\P_\gamma)$-generic filter $G_\gamma$ which contains $q*$. Thus, $q^*$ is $(N,\P_\gamma)$-generic.

To see that is $(N,\P_\gamma,E_S)$-generic. Let $b\in [S_\delta]^2$, $\dt{A}\in N$ a $\P_\gamma$-name for a subset of $[S]^2$ and a type $t:2\to\{<,>\}.$ Suppose $r\leq q^*$ and $r\forces_{\P_\gamma}``b\in \dt{A}".$ Fix $n$ such that $(b_n,\dt{A}_n,t_n)=(b,\dt{A},t).$ Now $r\leq p_{n+1}$, and thus $r$ forces the statement of clause (9).  This concludes the proof of the Theorem.

\end{proof}
\begin{corollary}\label{cor:main}
    Let $(S,<_s,<_{\lex})$ be a weakly bi-entangled Suslin  line. Let $\langle \P_\xi,\dt{\mathbb{Q}}_\eta:\xi\leq\gamma, \eta< \gamma\rangle  $ be a countable support iteration of $E_S$-proper forcing posets. Then $\P_\gamma$ is $E_S$-proper.
\end{corollary}

\section{Consistency results}\label{sec:final-model}
In this section we prove our main theorem. Recall the following notion.
\begin{definition}
    Let $\theta$ be a cardinal. We say that $\theta$ is \emph{supercompact} if for every cardinal $\lambda$ there exists a transitive inner model $M$ and an elementary embedding $j:V\to M$ such that:
    \begin{enumerate}[{\rm (1)}]
        \item crit$(j)=\theta$.
        \item $j(\theta)>\lambda$.
        \item $M^\lambda\subseteq M.$
    \end{enumerate}
\end{definition}
We will also need the following result of Laver.
\begin{theorem}
    Let $\theta$ be a supercompact cardinal. Then there is a function $f:\theta\to H(\theta)$ such that for every set $x$, and for every $\lambda\geq \theta$ satisfying $x\in H(\lambda)$ there exists an elementary embedding $j:V\to M$ (as a above) such that $j(f)(\theta)=x.$
\end{theorem}

\begin{theorem}

Assume that there is a supercompact cardinal. There is a model of $\ZFC+\OGA+\mathfrak{c}=\omega_2$ in which there is a $2$-entangled Suslin line.
\end{theorem}

\begin{proof}

First we start with a model of GCH in which there is a supercompact cardinal. Using Theorem \ref{thm:weaklybientangled}, we may further assume, that there is weakly bi-entangled lexicographically ordered  Suslin line $S$.
The rest of proof is nearly identical to the usual construction of Baumgartner of a model of the proper forcing axiom.
Fix a Laver function $f:\theta:\to H(\theta)$ which will be used as a bookkeeping device to anticipate all possible $S$-preserving and $E_S$-proper forcings.  We recursively construct a countable support iteration  $\langle \P_\xi,\dt{\mathbb{Q}}_\eta:\xi\leq\theta, \eta< \theta\rangle  $ as follows: At stage $\xi$, if $f(\xi)=(\dt{P},\dt{D})$ are $\P_\xi$-names such that $\dt{P}$ is forced to be $E_S$-proper and $S$-preserving and $\dt{\mathcal{D}}$ is forced to be a $\gamma$-sequence of dense open sets of $\dt{P}$ for some $\gamma<\theta,$ then we let $\dQ_\xi=\dt{P}$ and otherwise we let $\dQ_\xi$ is equal to the $\P_\xi$-name for the trivial forcing.

By Corollary \ref{cor:main} and Theorem \ref{miyamoto}, $\P_\theta$ is $E_S$-proper and $S$-preserving. In particular, $\omega_1$ is preserved. Also, standard arguments show that $\P_\theta$ is $\theta$-c.c., and by Lemma \ref{cor:E-proper-wbe}  $S$ is 2-entangled. Since $\omega_1$-closed forcing are $E_S$-proper and $S$-preserving (see Lemma~\ref{lem:s-closed-ET-proper}) and in particular the collapse Col$(\omega_1,\omega_2)$ is  $E_S$-proper and $S$-preserving, the iteration forces that $\omega_2=\theta.$ Finally, using the same elementary embedding argument as in the proof of the proper forcing axiom (see \cite[Theorem 31.21]{jech2003}), we can show that OGA holds in the extension.

\end{proof}
\section{Final remarks and open questions}

It is well-known that OGA can be forced without the use of large cardinals. So, the following question naturally arise.

\begin{problem}
    Is it possible to construct a model of OGA in which there is a 2-entangled linear order without large cardinals?
\end{problem}

\bibliographystyle{alpha}
\bibliography{references}

\end{document}